\documentclass[12pt,fleqn]{article}
\usepackage{amsmath, amsthm,enumerate,fleqn,amssymb,}
\usepackage{titlesec,soul}
\usepackage[margin=1in]{geometry}
\titleformat{\section}[block]{\centering\bfseries}{\thesection.}{1em}{}
\titlespacing*{\section}{36pt}{6pt}{6pt}
\titleformat{\subsection}[block]{\centering\bfseries}{\thesection.}{1em}{}
\titlespacing*{\subsection}{24pt}{6pt}{6pt}
\newtheoremstyle{theorem}{10pt}{10pt}{\sl}{\parindent}{\bf}{. }{ }{}
\theoremstyle{theorem}
\newtheorem{theorem}{Theorem}[section]

\newtheorem{lemma}[theorem]{Lemma}
\newtheoremstyle{defi}{10pt}{10pt}{\rm}{\parindent}{\bf}{. }{ }{}
\theoremstyle{defi}

\newtheorem{example}[theorem]{Example}
\newtheorem{remark}[theorem]{Remark}

\numberwithin{equation}{section}
\date{}
\newcommand{\al}{\alpha}
\newcommand{\be}{\beta}

\newcommand{\fr}{\mathcal{F}}

\newcommand{\ity}{\infty}
\newcommand{\C}{\mathbb{C}}

\newcommand{\f}{f^n(f^{n_1})^{(t_1)}\ldots(f^{n_k})^{(t_k)}}

\begin{document}
\title{\normalsize\bf SOME NORMALITY CRITERIA}
\author{\footnotesize  GOPAL DATT$^1$ AND SANJAY KUMAR$^2$  \\
\footnotesize  $^1$Department of Mathematics,\\[-2pt]
\footnotesize  University of Delhi, Delhi 110007, India\\[-2pt]
\footnotesize  e-mail: ggopal.datt@gmail.com\\
\footnotesize   $^2$Department of Mathematics,\\[-2pt]
\footnotesize  Deen Dayal Upadhyaya College, University of Delhi, Delhi 110015, India\\[-2pt]
\footnotesize  e-mail: sanjpant@gmail.com}
%\footnotesize   $^3$Department of Mathematics,\\[-2pt]
%\footnotesize  Kirori Mal college, University of Delhi, Delhi 110007, India\\[-2pt]
%\footnotesize  e-mail: shikk2003@yahoo.co.in
\maketitle
\begin{center}
\begin{minipage}{\textwidth}\fontsize{10}{10}\selectfont
\textbf{Abstract:}\ In this article we prove some normality criteria for a family of meromorphic functions which involves sharing of a non-zero value  by certain  differential monomials generated by the members of the family. These result generalize some of the results of Schwick.
\end{minipage}
\end{center}

\jot6pt
\vskip3em
\thispagestyle{empty}
\baselineskip14pt\fontsize{11}{14}\selectfont
\section{Introduction and main results}
The notion of normal families was introduced by Paul Montel in 1907. Let us begin by recalling the definition. A family   of meromorphic functions defined on a domain $D\subset \C$ is said  to be normal in the domain, if every sequence in the family  has a subsequence which converges spherically  uniformly on compact subsets of  $D$ to a meromorphic function or to $   \infty$ (see \cite{Ahl, Hay, Schiff, Yang}). \\

One important aspect  of the theory of complex analytic functions is to find normality criteria for families of meromorphic functions.  Montel obtained a normality criterion, now known as The Fundamental Normality Test, which  says that  {\it a family of meromorphic functions in a domain  is normal if it omits three distinct complex numbers.} This  result has undergone  various extensions. In 1975,  Lawrence Zalcman \cite{Zalc 1} proved a remarkable  result, now known as Zalcman's Lemma, for  families of meromorphic functions which are not normal in a domain. Roughly speaking, it says that {\it a non-normal family can be rescaled at small scale to obtain a non-constant meromorphic function in the limit. } This result of Zalcman gave birth to many new normality criteria. These normality criteria have been used extensively in complex dynamics for studying the  Julia-Fatou dichotomy.\\

Wilhelm Schwick ~\cite{Sch} gave a connection between normality and sharing values and proved a result  which says that  {\it{a family  of meromorphic functions on a domain $D\subset \C$ is normal if every function of the family and its first order derivative share three distinct complex numbers.}} Since then many results of normality criteria concerning sharing values have been obtained ~\cite{GS, Li, YMW, lahiri, lahiri 1}.\\

Let $f$ and $g$ be  meromorphic functions in a domain $D$ and  $p\in\mathbb{C}.$ If the   zeros of $f-p$ are the zeros of  $g-p$  ignoring multiplicity, we write $f=p\Rightarrow g=p.$ Hence $f=p\iff g=p$ means that $f-p$ and $g-p$ have the same zeros ignoring multiplicity. If $f-p=0\iff g-p=0$,   then we say that $f$ and $g$ share the value $p$ IM (see \cite{ccyang}).\\

Schwick \cite{Sch1} also proved a normality criterion which states that: {\it{Let $n, k$ be positive integers such that  $n\geq k+3$, let $\mathcal F$ be a family of functions meromorphic  in $ D$.  If each $f\in \mathcal F$ satisfies $(f^n)^{(k)}(z)\neq 1$ for $z\in  D$, then $\mathcal F$ is a normal family.}} This result holds good for holomorphic functions in case $n\geq k+1$. Recently  Gerd Dethloff et al. \cite{dethloff} came up with new normality criteria, which improved the result given by Schwick \cite{Sch1}.
 \begin{theorem}\label{thm}
 Let $p\neq 0$ be a complex number,  $n$ be  a non-negative integer and $n_1, n_2,\ldots, n_k,$ $ t_1, t_2,\ldots, t_k$ be positive integers. Let $\fr$ be a family of meromorphic functions in a  domain $D$ such that for every $f\in \fr, f^n(f^{n_1})^{(t_1)}\ldots (f^{n_k})^{(t_k)}- p $ is nowhere vanishing on $D$. Assume that
 \begin{enumerate}
 \item[$(a)$]{ $n_j \geq t_j$ for all $1\leq j \leq k$,}
 \item [$(b)$]{$n+\sum_{j=1}^{k}n_j\geq 3+ \sum_{j=1}^{k}t_j.$}
 \end{enumerate}
  Then $\fr$ is normal on $D$.
 \end{theorem}

 For the case of holomorphic functions they proved the following strengthened version:
 \begin{theorem}
 Let $p\neq 0$ be a complex number, let $n$ be a non-negative integer and $n_1, n_2,\ldots, n_k,$  $t_1, t_2,\ldots, t_k$ be positive integers. Let $\fr$ be a family of holomorphic functions in a domain $D$ such that for every $f\in \fr, f^n(f^{n_1})^{(t_1)}\ldots (f^{n_k})^{(t_k)}-p$ is nowhere vanishing on $D$. Assume that
 \begin{enumerate}
 \item[$(a)$]{ $n_j \geq t_j$ for all $1\leq j \leq k,$}
 \item [$(b)$]{$n+\sum_{j=1}^{k}n_j\geq 2 + \sum_{j=1}^{k}t_j.$}
 \end{enumerate}
  Then $\fr$ is normal on $D$.
 \end{theorem}

 The main aim of this paper is to find  normality criteria in terms of sharing values which is motivated by \cite{dethloff}.

\begin{theorem}\label{thm1}
 Let $ p\neq 0 $ be a complex number,  $n$ be  a non-negative integer and $n_1, n_2,\ldots, n_k,$ $ t_1, t_2,\ldots, t_k $ be  positive integers such that
  \begin{enumerate}
 \item[$(a)$]{ $n_j \geq t_j$ for all $1\leq j \leq k,$}
 \item [$(b)$]{$n+\sum_{j=1}^{k}n_j\geq 3+ \sum_{j=1}^{k}t_j.$}
 \end{enumerate}
Let $\fr$ be a family of meromorphic functions in a domain $D$ such that for every pair of functions $f, g \in \fr,  f^n(f^{n_1})^{(t_1)}\ldots(f^{n_k})^{(t_k)}$ and $g^n(g^{n_1})^{(t_1)}\ldots(g^{n_k})^{(t_k)}$ share $p$ IM on $D$. Then $\fr$ is normal in $D$.
 \end{theorem}

 For  families of holomorphic functions we have the following strengthened version:
\begin{theorem}\label{thm4}
Let $ p\neq 0 $ be a complex number, $n$ be a non-negative integer and $n_1, n_2,\ldots, n_k,$ $ t_1, t_2,\ldots, t_k $ be  positive integers such that
   \begin{enumerate}
 \item[$(a)$]{ $n_j \geq t_j$ for all $1\leq j \leq k,$}
 \item [$(b)$]{$n+\sum_{j=1}^{k}n_j\geq 2+ \sum_{j=1}^{k}t_j.$}
 \end{enumerate}
Let $\fr$ be a family of holomorphic functions in a domain $D$ such that for every pair of functions $f, g \in \fr,  f^n(f^{n_1})^{(t_1)}\ldots(f^{n_k})^{(t_k)}$ and $g^n(g^{n_1})^{(t_1)}\ldots(g^{n_k})^{(t_k)}$ share $p$ IM on $D$. Then $\fr$ is normal in $D$.
 \end{theorem}
The following examples show that the condition on $p$ is necessary.
\begin{example}
Let $\fr=\{e^{mz}: m= 1, 2, \ldots\}$ be a family on $\Delta:=\{z:|z|<1\}.$ Let $n, n_i's$ and $ t_i's$ be as in Theorem \ref{thm1}. Then for every pair $f, g \in \fr, $ $ f^n(f^{n_1})^{(t_1)}\ldots(f^{n_k})^{(t_k)}$ and $g^n(g^{n_1})^{(t_1)}\ldots(g^{n_k})^{(t_k)}$ share 0 and $\infty$. But $\fr$ is not normal.
\end{example}
\begin{example}
  Let $\fr=\{mz: m= 1, 2, \ldots\}$ be a family on $\Delta:=\{z:|z|<1\}.$ Let $n, n_i's$ $ t_i's$ be as in Theorem \ref{thm1}. Then for every pair $f, g \in \fr,$ $  f^n(f^{n_1})^{(t_1)}\ldots(f^{n_k})^{(t_k)}$ and $g^n(g^{n_1})^{(t_1)}\ldots(g^{n_k})^{(t_k)}$ share 0 and $\infty$. But $\fr$ is not normal.
\end{example}
The following example supports our result.
\begin{example}
Let $\fr=\{f_n: n\in \mathbb N\},$ where $f_n(z)=n $. Then $\fr$ satisfies conditions of Theorem \ref{thm1} and $\fr$ is normal.
\end{example}
It is natural to ask what  happens if we have a zero of $f^n(f^{n_1})^{(t_1)}\ldots(f^{n_k})^{(t_k)}- p$. For this question we can extend Theorem \ref{thm} in the following manner.
\begin{theorem}\label{thm2}
Let $ p\neq 0 $ be a complex number,  $n$ be  a non-negative integer and $n_1, n_2,\ldots, n_k,$ $ t_1, t_2,\ldots, t_k $ be  positive integers such that
   \begin{enumerate}
 \item[$(a)$]{ $n_j \geq t_j$ for all $1\leq j \leq k,$}
 \item [$(b)$]{$n+\sum_{j=1}^{k}n_j\geq 3+ \sum_{j=1}^{k}t_j.$}
 \end{enumerate}
Let $\fr$ be a family of meromorphic functions in a domain $D$ such that for every $f\in\fr, f^n(f^{n_1})^{(t_1)}\ldots(f^{n_k})^{(t_k)}- p$ has at most one zero IM. Then $\fr$ is normal in $D$.
\end{theorem}

\begin{remark}
Theorem \ref{thm} is an immediate corollary of Theorem \ref{thm1} and Theorem \ref{thm2}.
\end{remark}

\section{Some Notations}
 Let $\Delta=\{z: |z|<1\}$ be the unit disk and $\Delta(z_0, r):=\{z: |z-z_0|<r\}.$ We use the following standard functions of value distribution theory, namely
\begin{center}
$T(r,f),  m(r,f),  N(r,f)\ \text{and}\ \overline{N}(r,f)$.
\end{center}
We let $S(r,f)$ be any function satisfying
\begin{center}
$S(r,f)=o\big(T(r,f)\big)$,  as $r\rightarrow +\ity,$
\end{center}
 possibly outside of a set with finite measure.
\section{Some Lemmas}

 In order to prove our results we need  the following Lemmas. The following is a new version of Zalcman's Lemma (see \cite{Zalc 1, Zalc}).

\begin{lemma} \label{lem1}Let $\mathcal F$ be a family of meromorphic  functions in the unit disk  $\Delta$, with the property that for every function $f\in \mathcal F,$  the zeros of $f$ are of multiplicity at least $l$ and the poles of $f$ are of multiplicity at least $k$ . If $\mathcal F$ is not normal at $z_0$ in $\Delta$, then for  $-l< \alpha <k$, there exist
\begin{enumerate}
\item{ a sequence of complex numbers $z_n \rightarrow z_0$, $|z_n|<r<1$},
\item{ a sequence of functions $f_n\in \mathcal F$, }
\item{ a sequence of positive numbers $\rho_n \rightarrow 0$},
\end{enumerate}
such that $g_n(\zeta)=\rho_n^{\alpha}f_n(z_n+\rho_n\zeta) $ converges to a non-constant meromorphic function $g$ on $\C$ with $g^{\#}(\zeta)\leq g^{\#}(0)=1$. Moreover $g$ is of order at most two. Here $g^{\#}$ denote the spherical derivative of $g$.
\end{lemma}
\begin{lemma}\cite{CH}\label{lemmaCH}
Let $f$ be an entire function. If the spherical derivative $f^{\#}(z)$ is bounded for all $z\in\C$, then $f$ has order at most 1.
\end{lemma}

Let $f$ be a non-constant meromorphic function in $\C$. A  differential polynomial $P$ of $f$ is defined by $\displaystyle {P(z):= \sum_{i=1}^{n}\al_{i}(z)\prod_{j=0}^{p}\left(f^{(j)}\left(z\right)\right)^{S_{ij}}},$ where $S_{ij}$'s are nonnegative integers and    $\al_i(z)\not\equiv 0$ are small functions of $f$, which means $T(r,\al_i)=o\big(T(r,f)\big)$. The lower degree of the differential polynomial $P$ is defined by $$d(P):= \min_{1\leq i \leq n}\sum_{j=0}^{p}S_{ij} .$$

The following result was proved by Dethloff et al. in ~\cite{dethloff}.
\begin{lemma}\label{lem2}
Let $a_1, \ldots, a_q$ be distinct non-zero complex numbers.  Let $f$ be a non-constant  meromorphic function and let $P$  be a non-constant  differential polynomial of $f$ with $d(P)\geq 2.$ Then
\begin{equation}\notag
T(r,f)\leq \left(\frac{q\theta(P)+1}{qd(P)-1}\right)\overline{N}\left(r,\frac{1}{f}\right)+\frac{1}{qd(P)-1}\sum_{j=1}^{q}\overline{N}\left(r,\frac{1}{P-a_j}\right)+ S\left(r,f\right),
\end{equation}
for all $r\in [1,+\ity)$ excluding a set of finite Lebesgue measure, where $\displaystyle\theta(P):= \max_{1\leq i \leq n}\sum_{j=0}^{p}jS_{ij}.$\\

 Moreover, in the case of an entire function, we have
 \begin{equation}\notag
T(r,f)\leq \bigg(\frac{q\theta(P)+1}{qd(P)}\bigg)\overline{N}\left(r,\frac{1}{f}\right)+\frac{1}{qd(P)}\sum_{j=1}^{q}\overline{N}\left(r,\frac{1}{P-a_j}\right)+ S(r,f),
\end{equation}
for all $r\in [1,+\ity)$ excluding a set of finite Lebesgue measure.
\end{lemma}

This result was proved  by Hinchliffe in ~\cite{hin} for $q=1$.
\begin{lemma}\label{lem3}
Let $f$ be a transcendental meromorphic function. Let $n$ be a non-negative integer and $n_1, n_2, \ldots, n_k, t_1, t_2, \ldots, t_k  $ be positive integers such that
\begin{enumerate}
\item[$(a)$] {$n_j \geq t_j$ for all $1 \leq j \leq k,$}
\item[$(b)$] {$n+ \sum _{j=1}^{k}n_j\geq 3 + \sum_{j=1}^{k}t_j$.}
\end{enumerate}
Then $\f$ assumes every non-zero complex value $p\in\C$ infinitely often.
\end{lemma}
\begin{proof}
On the contrary, assume that $\f$ takes the value $p$ only finitely many times. Then
\begin{equation}\label{eq}
N\left(r, \frac{1}{(\f-p}\right) = O\left(\log r\right)= S(r, f).
\end{equation}

Without loss of generality we may assume $p=1$. Let $P=\f$. Consider $\displaystyle(f^{n_i})^{(t_i)}=\sum c_{m_0,m_1,\ldots,m_{t_i}}f^{m_0}(f')^{m_1}\ldots (f^{(t_i)})^{m_{t_i}}$, where $c_{m_0,m_1,\ldots,m_{t_i}}$ are constants and $m_0, m_1,\ldots, m_{t_i}$ are non-negative integers such that  $\displaystyle\sum_{j=0}^{t_i}m_j = n_i, \sum_{j=1}^{t_i}jm_j=t_i.$ It is easy to calculate
$$d(P) = n + \sum_{j=1}^{k}n_j \ \text{and}\ \theta(P)=\sum_{j=1}^{k}t_j.$$
 Clearly $d(P)>2.$ So by Lemma \ref{lem2} we get
\begin{equation}\notag
T(r,f)\leq \bigg(\frac{\sum_{j=1}^k t_j + 1}{n + \sum_{j=1}^{k}n_j - 1}\bigg)\overline{N}\left(r,\frac{1}{f}\right)+\bigg(\frac{1}{n + \sum_{j=1}^{k}n_j -  1}\bigg)\overline{N}\left(r,\frac{1}{P-1}\right)+ S(r,f),
\end{equation}
and this gives
\begin{equation}\notag
\bigg(\frac{n + \sum_{j=1}^{k}n_j-\sum_{j=1}^k t_j-2}{n + \sum_{j=1}^{k}n_j-1}\bigg)T(r,f)\leq \bigg(\frac{1}{n + \sum_{j=1}^{k}n_j -  1}\bigg)\overline{N}\left(r,\frac{1}{P-1}\right)+ S(r,f),
\end{equation}
and this gives
\begin{equation}\notag
\bigg(\frac{1}{n + \sum_{j=1}^{k}n_j-1}\bigg)T(r,f)\leq \bigg(\frac{1}{n + \sum_{j=1}^{k}n_j -  1}\bigg)N\left(r,\frac{1}{P-1}\right)+ S(r,f).
\end{equation}
 By using \eqref{eq} we get $T(r,f)=S(r,f)$, which is a contradiction.
\end{proof}

\begin{lemma}\label{lem5}
Let $f$ be a transcendental entire function. Let $n$ be a non-negative integer and $n_1, n_2, \ldots, n_k, t_1, t_2, \ldots, t_k  $ be positive integers such that
\begin{enumerate}
\item[$(a)$] {$n_j \geq t_j$ for all $1 \leq j \leq k,$}
\item[$(b)$] {$n+ \sum _{j=1}^{k}n_j\geq 2 + \sum_{j=1}^{k}t_j$.}
\end{enumerate}
Then $\f$ assumes every non-zero complex value $p\in\C$ infinitely often.
\end{lemma}
We can prove this lemma by arguments similar to the proof of Lemma \ref{lem3}.

\begin{lemma}\label{lemma2}\cite{lahiri, lahiri 1},
Let $R=\frac{P}{Q} $ be a rational function and $Q$ be non-constant. Then $\left(R^{(k)}\right)_{\ity}\leq(R)_{\ity}-k,$ where $k$ is a positive integer,  $(R)_{\ity}=$ $\text{deg}(P)-\text{deg}(Q)$ and $\text{deg}(P)$ denotes the degree of P.
\end{lemma}

\begin{lemma}\label{lemma3}\cite{lahiri}
Let $R = a_mz^m + \ldots + a_1z + a_0 +\frac{P}{B},$ where $a_0, a_1, \ldots, a_{m-1}, a_m(\neq0) $ are constants, $m$ is a positive integer and $P,\  B$ are polynomials with deg$(P)<$ deg$(B)$. If $k\leq m$,  then $\left(R^{(k)}\right)_{\ity} = (R)_{\ity}-k,$
\end{lemma}

\begin{lemma}\label{lem4}
Let $f$ be a non-constant rational  function, $p\in\C\setminus\{0\}$,  $n$ be a non-negative integer and $n_1, n_2, \ldots, n_k, t_1, t_2, \ldots, t_k  $ be positive integers such that
\begin{enumerate}
\item[$(a)$] {$n_j \geq t_j$ for all $1 \leq j \leq k,$}
\item[$(b)$] {$n+ \sum _{j=1}^{k}n_j\geq 3 + \sum_{j=1}^{k}t_j$.}
\end{enumerate}
Then $\f$ has at least two distinct $p$-points.
\end{lemma}
\begin{proof}
 On the contrary, assume that $\f$ has at most one $p$-point. Now there are two cases to consider.\\

\underline{Case 1:} Suppose $\f$ has exactly one $p$-point.
First we assume that $f$ is a non-constant  polynomial. Since $\f$ has exactly one $p$-point, we can set
$$\f-p=A(z-z_0)^l,$$
  where $A$ is a non-zero constant and $l$ is a positive integer satisfying $l\geq n+\sum n_j -\sum t_j\geq 3.$ Then $$\left(\f\right)' = Al(z-z_0)^{l-1}.$$ Since a zero of $f$ is a zero of $\f$ with multiplicity greater than $1$, so it is also a zero of $\left(\f\right)'.$ Since $\left(\f\right)'$ has exactly one zero namely $ z_0$ and $f$  is a non-constant polynomial, it follows that $z_0$ is a zero of $f$ and so is a zero of $\f$, which is a contradiction. Therefore $f $ is a rational function which is not a polynomial. Let
  \begin{equation}\label{eq2.1}
f(z)=A\frac{(z-\al_1)^{m_1}(z-\al_2)^{m_2}\ldots (z-\al_s)^{m_s}}{(z-\be_1)^{n_1'}(z-\be_2)^{n_2'}\ldots (z-\be_t)^{n_t'}},
\end{equation}
where $A$ is non-zero constant and $m_i$'s  and
$n_j$'s are integers. We put
\begin{equation}\notag
M=n\sum_{i=1}^{s}m_i,
  N=n\sum_{j=1}^{t}n_j',
 \end{equation}
 and
 \begin{equation}\notag
  M_i=n_i\sum_{j=1}^{s}m_j,
   N_i=n_i\sum_{j=1}^{t}n_j', i=1, 2,\ldots, k.
\end{equation}

From \eqref{eq2.1} we get
\begin{equation}\label{1}
f^{n_i}(z)=A^{n_i}\frac{(z-\al_1)^{n_im_1}(z-\al_2)^{n_im_2}\ldots (z-\al_s)^{n_im_s}}{(z-\be_1)^{n_in_1'}(z-\be_2)^{n_in_2'}\ldots (z-\be_t)^{n_in_t'}},
 \end{equation}
 and so
 \begin{equation}\label{2}
 (f^{n_i})^{(t_i)}(z)=\frac{(z-\al_1)^{n_im_1-t_i}(z-\al_2)^{n_im_2-t_i}\ldots (z-\al_s)^{n_im_s-t_i}g_i(z)}{(z-\be_1)^{n_in_1'+t_i}(z-\be_2)^{n_in_2'+t_i}\ldots (z-\be_t)^{n_in_t'+t_i}},
  \end{equation}
  where $g_i(z)$ is a polynomial. From \eqref{1} and \eqref{2} we get
  \begin{center}
  $(f^{n_i})_{\ity} = M_i - N_i \ \text{and}\ \left((f^{n_i})^{(t_i)}\right)_{\ity}  = M_i-N_i-t_i(s+t)+\text{deg}g_i(z)$.
  \end{center}
  Since by Lemma \ref{lemma2} $\left((f^{n_i})^{(t_i)}\right)_{\ity}\leq (f^{n_i})_{\ity} - t_i$, we get
  \begin{equation}
  \text{deg} (g_i)\leq t_i(s+t-1).
  \end{equation}

  From \eqref{eq2.1} and \eqref{2} we get
\begin{align}
f^n(f^{n_1})^{(t_1)}\ldots (f^{n_k})^{(t_k)}&\label{eq2.2}\\
&=A^{n}\frac{(z-\al_1)^{m_1n'-t'}(z-\al_2)^{m_2n'-t'}\ldots (z-\al_s)^{m_sn'-t'}g(z)}{(z- \be_1)^{n_1'n'+t'}(z- \be_2)^{n_2'n'+t'}\ldots (z- \be_t)^{n_t'n'+t'}}\notag\\
&=\frac{p_1}{q_1},  \notag
\end{align}
where  $\displaystyle n'= n+\sum_{j=1}^{k}n_j$, $\displaystyle t'=\sum_{j=1}^{k}t_j$ and  $p_1, q_1 $, $g(z)$ are polynomials with
\begin{equation}\label{3}
\text{deg}\left(g(z)\right)\leq (s+t-1)\sum_{j=1}^{k}t_j= t'(s+t-1).
 \end{equation}

Since $\f$ has exactly one $p-$point and it is at $z_0,$ we get from \eqref{eq2.2} that
\begin{align}
f^n(f^{n_1})^{(t_1)}\ldots (f^{n_k})^{(t_k)}&\notag\\
&=p+\frac{B(z-z_0)^l}{(z-\be_1)^{n_1'n'+t'}(z-\be_2)^{n_2'n'+t'}\ldots (z-\be_t)^{n_t'n'+t'}}\notag\\
&=\frac{p_1}{q_1}, \label{eq2.3}
\end{align}
where $B$ is a non-zero constant and $l$ is  a positive integer. From \eqref{eq2.2} we also obtain that
\begin{align}
\left(\f\right)'&\notag\\
&=\frac{(z-\al_1)^{m_1n'-t'-1}(z-\al_2)^{m_2n'-t'-1}\ldots (z-\al_s)^{m_sn'-t'-1}g_1(z)}{(z-\be_1)^{n_1'n'+t'+1}(z-\be_2)^{n_2'n'+t'+1}\ldots (z-\be_t)^{n_t'n'+t'+1}},\label{eq2.4}
\end{align}
where $g_1(z)$ is a polynomial. From \eqref{eq2.3} we obtain that
\begin{align}
\left(f^n(f^{n_1}\right)^{(t_1)}\ldots (f^{n_k})^{(t_k)})'&\notag\\
&=\frac{(z-z_0)^{l-1}g_2(z)}{(z-\be_1)^{n_1'n'+t'+1}(z-\be_2)^{n_2'n'+t'}\ldots (z-\be_t)^{n_t'n'+t'+1}}, \label{eq2.5}
\end{align}
where $g_2(z)$ is a polynomial. From \eqref{eq2.2} and \eqref{eq2.4} we obtain
\begin{align*}
\left(\f\right)_{\ity}&= M + \sum_{i=1}^{k}M_i - st'+\text{deg}(g(z))-N-\sum_{i=1}^{k}N_i-tt',\\
\left(\left(\f\right)'\right)_{\ity}&= M+ \sum_{i=1}^{k}M_i - st'+\text{deg}\left(g_1(z)\right)-N-\sum_{i=1}^{k}N_i-tt'-s-t.
\end{align*}

By Lemma \ref{lemma2} we get
\begin{equation}
\left(\left(\f\right)'\right)_{\ity}\leq \left(\f\right)_{\ity}-1.
\end{equation}

Hence we obtain
\begin{align}
\text{deg}\left( g_1\left(z\right)\right)&\leq s+t+\text{deg}(g(z))-1\notag\\
&\leq s+t+(s+t-1)t' -1\notag\\
&= (s+t-1)(t'+1).\label{eq2.6}
\end{align}

%Again from \eqref{eq2.3} and \eqref{eq2.5} we get
%\begin{align}
%\left(\f\right)_{\ity}&=l+\left(N+\sum_{i=1}^{k} N_i+tt'\right) \ \text{and}\notag\\
%\left(\left(\f\right)'\right)_{\ity}&=l-1+ \ \text{deg}(g_2)-\left(\sum_{i=1}^{t}n_i'\right)n'+tt' +t),\notag
%\end{align}
%and from this with  Lemma \ref{lemma2} we obtain deg$(g_2)\leq t.$

Now we consider the following subcases.\\

  \underline{Subcase 1.} When  $\displaystyle l < N + \sum_{i=1}^{k} N_i +tt' $.\\

   From \eqref{eq2.3} we have $\text{deg}(p_1)=\text{deg}(q_1)$, also  from \eqref{eq2.2} and \eqref{3} we get that
  \begin{align}
  \text{deg}(q_1)&=N + \sum_{i=1}^{k} N_i + tt' = \text{deg}(p_1)\notag\\
  &\leq M + \sum_{i=1}^{k} M_i + (t-1)t'.\notag
  \end{align}

Hence $\left(M+\sum_{i=1}^{k} M_i\right)-(N+\sum_{i=1}^{k} N_i)\geq t'. $ This implies $\sum_{j=1}^{s} m_j - \sum_{j=1}^{t}n_j'\geq 1.$ Therefore $(f)_{\ity}\geq1$ and $(f^{n_i})_{\ity}\geq n_i$. Therefore we can write $f^{n_i}$ as follows
\begin{center}
$f^{n_i} = a_mz^m +\ldots +a_1z + a_0 +\frac{p}{B},$
\end{center}
where  $m (\geq n_i)$ is an integer, $a_m, \ldots, a_1, a_0$ are constants such that $a_m\neq0$ and $p, \  B $ are polynomials with $\text{deg}(p)< \text{deg}(B).$ Now by using Lemma \ref{lemma3} we get
\begin{equation}\label{eq2.7}
\left(\left(f^{n_i}\right)^{(t_i)}\right)_{\ity}= (f^{n_i})_{\ity} -t_i \geq n_i -t_i.
\end{equation}

Since $(f)_{\ity}\geq1$, from \eqref{eq2.7} we see that $\left(\f\right)_{\ity}\geq n'-t'\geq 3$,  which contradicts the fact that deg$(p_1)= \text{deg} (q_1)$.\\

\underline{Subcase 2.}  When $\displaystyle l = N + \sum_{i=1}^{k} N_i +tt' $.\\

Then from \eqref{eq2.3} we get $\left(\f\right)_{\ity}\leq 0$. Now we show that $$\sum_{i=1}^{s}m_i\leq\sum_{i=1}^{t}n_i'. $$ Otherwise $(f^n)_{\ity}= n\sum_{i=1}^{s} m_i -n \sum_{i=1}^{t} n_i'\geq n$ and $\left(\left(f^{n_i}\right)^{(t_i)}\right)_{\ity}=(f^{n_i})_{\ity}-t_i\geq n_i-t_i$ and so $\left(\f\right)_{\ity}\geq n+\sum_{i=1}^{k} n_i- \sum_{i=1}^{k}t_i\geq 3$, which is a contradiction.\\

Since $\al_i\neq z_0$ for $i=1, 2,\ldots, s$ from \eqref{eq2.4} and \eqref{eq2.5} we see that $(z-z_0)^{l-1} $ is a factor of $g_1$. Therefore by \eqref{eq2.6} we get $l-1\leq \text{deg}(g_1)\leq (s+t-1)(t'+1).$ Now we have
\begin{align*}
N+\sum_{i=1}^{k}N_i&= l-t\sum_{i=1}^{k}t_i\\
&\leq (s+t-1)\bigg(\sum_{i=1}^{k}t_i+1\bigg)+1-t\sum_{i=1}^{k}t_i\\
&=s\bigg(\sum_{i=1}^{k}t_i+1\bigg)+t-\sum_{i=1}^{k}t_i\\
&\leq \sum_{i=1}^{s}m_i\bigg(\sum_{i=1}^{k}n_i+1\bigg) + \sum_{i=1}^{t}n_i'-\sum_{i=1}^{k}t_i\\
&\leq \sum_{i=1}^{k}M_i+2\sum_{i=1}^{t}n_i'-\sum_{i=1}^{k}t_i\\
&\leq \sum_{i=1}^{k}N_i+2\sum_{i=1}^{t}n_i'-\sum_{i=1}^{k}t_i,
\end{align*}
which is a contradiction when $n>2$. For the case $0\leq n\leq2$ we use the condition $n+\sum_{i=1}^{k}n_i\geq 3+\sum_{i=1}^{k}t_i,$ to get
\begin{align*}
N+\sum_{i=1}^{k}N_i&\leq \sum_{i=1}^{s}m_i\bigg(\sum_{i=1}^{k}t_i+1\bigg)+t-\sum_{i=1}^{k}t_i\\
&\leq \sum_{i=1}^{k}n_i \sum_{i=1}^{s}m_i + \sum_{i=1}^{t}n_i'-\sum_{i=1}^{k}t_i\\
&\leq\sum_{i=1}^{k}M_i + \sum_{i=1}^{t}n_i'-\sum_{i=1}^{k}t_i\\
&\leq\sum_{i=1}^{k}N_i+\frac{N}{n}-\sum_{i=1}^{k}t_i\\
&\leq N +\sum_{i=1}^{k}N_i-\sum_{i=1}^{k}t_i,
\end{align*}
which is again a contradiction.\\

\underline{Subcase 3.} When $\displaystyle l > N + \sum_{i=1}^{k} N_i +tt' $.\\

Then from \eqref{eq2.3} we have $(\f)_{\ity}> 0$. Now we claim $$\sum_{i=1}^{s}m_i>\sum_{i=1}^{t}n_i'.$$ If $\sum_{i=1}^{s}m_i\leq\sum_{i=1}^{t}n_i'$, then $(f)_{\ity}\leq0, (f^{n_i})_{\ity}\leq0 \ \text{and} \ (f^{n})_{\ity}\leq0 $. Hence by Lemma \ref{lemma2} we obtain that
\begin{align*}
\left(\f\right)_{\ity}&=(f^n)_{\ity}+(f^{n_1})_{\ity}+\ldots +(f^{n_k})_{\ity}\\
&\leq 0+\sum_{i=1}^{\ity}(f^{n_i})_{\ity}-t_i<0,
\end{align*}
which is a contradiction.\\

Again from \eqref{eq2.3} and \eqref{eq2.5} we get
\begin{align}
\left(\f\right)_{\ity}&=l-\left(N+\sum_{i=1}^{k} N_i+tt'\right) \ \text{and}\notag\\
\left(\left(\f\right)'\right)_{\ity}&=l-1+ \ \text{deg}(g_2)-\left(\sum_{i=1}^{t}n_i'\right)n'+tt' +t),\notag
\end{align}
and from this with  Lemma \ref{lemma2} we obtain deg$(g_2)\leq t.$

Since for each $i=1, 2, \ldots, s, \al_i\neq z_0.$ From \eqref{eq2.4} and \eqref{eq2.5} we observe  that,\\  $\displaystyle{(z-\al_1)^{m_1n'-t'-1}(z-\al_2)^{m_2n'-t'-1}\ldots (z-\al_s)^{m_sn'-t'-1}}$ is a factor of $g_2.$ Therefore
\begin{equation}\label{4}
M+\sum_{i=1}^{k}M_i-st'-s \leq \ \text{deg}(g_2)\leq t,
\end{equation}

and from \eqref{4} we get that
\begin{align}
M+\sum_{i=1}^{k}M_i&\leq s+t +st'\notag\\
&=t+(t'+1)s\notag\\
&\leq \sum_{i=1}^{t}n_i'+\left(\sum_{i=1}^{k}n_i+1\right)\sum_{i=1}^{s}m_i\notag\\
&< \sum_{i=1}^{s}m_i+\left(\sum_{i=1}^{k}n_i+1\right)\sum_{i=1}^{s}m_i\notag\\
&=\frac{2}{n}M +\sum_{i=1}^{k} M_i,\notag
\end{align}
which is a contradiction when $n>2.$ For the case $0\leq n\leq2$ we use the condition $ n+ \sum_{i=1}^{k}n_i\geq 3+\sum_{i=1}^{k}t_i,$ to get
\begin{align*}
M+\sum_{i=1}^{k}M_i&\leq \sum_{i=1}^{t}n_i'+\left(\sum_{i=1}^{k}t_i+1\right)\sum_{i=1}^{s}m_i\\
&\leq \frac{N}{n_1}+ \sum_{i=1}^{k}n_i\sum_{i=1}^{s}m_i\\
&< \frac{M}{n_1}+ \sum_{i=1}^{k}M_i,
\end{align*}
which is a contradiction.\\

\underline{Case 2.}\ Suppose   $\f-p$ has no zero. Then $f$ cannot be a polynomial. So $f$ is a rational function which is not a polynomial. Now we put $l=0$ in \eqref{eq2.3} and proceed as  in Subcase 1.\\
\end{proof}
\section{Proof of Main Results}
{\bf{Proof of Theorem \ref{thm1}}}:  Since normality is a local property, we assume that
${D}=\Delta$. Suppose that $\fr$ is not normal
in $\Delta$. Then there exists at least one point $z_0$ such that
$\fr$ is not normal  at the point $z_0$ in $\Delta$. Without loss of
generality we assume that $z_0=0$. Then by Lemma \ref{lem1}, for \[\al=-\frac{\sum_{i=1}^{k}t_i}{n+\sum_{i=1}^{k}n_i}\]  there
exist
\begin{enumerate}
\item{ a sequence of complex numbers $z_j \rightarrow 0$, $|z_j|<r<1$},
\item{ a sequence of functions $f_j\in \mathcal F$ },
\item{ a sequence of positive numbers $\rho_j \rightarrow 0$},
\end{enumerate}
such that $g_j(\zeta)=\rho_j^{\alpha}f_j(z_j+\rho_j\zeta) $ converges to a non-constant meromorphic function $g(\zeta)$ on $\C$ with $g^{\#}(\zeta)\leq g^{\#}(0)=1$. Moreover $g$ is of order at most two . \\
We see that
\begin{align}
 f_j^n(z_j+\rho_j\zeta)&(f_j^{n_1})^{(t_1)}(z_j+\rho_j\zeta)\ldots(f_j^{n_k})^{(t_k)}(z_j+\rho_j\zeta)\notag\\
&=g_j^n(\zeta)(g_j^{n_1})^{(t_1)}(\zeta)\ldots(g_j^{n_k})^{(t_k)}(\zeta)\rightarrow g^n(\zeta)(g^{n_1})^{(t_1)}(\zeta)\ldots(g^{n_k})^{(t_k)}(\zeta),\label{4.1}
\end{align}
as $j\rightarrow \ity$, locally spherically uniformly.\\

Let
\begin{equation}\label{4.2}
g^n(\zeta)(g^{n_1})^{(t_1)}(\zeta)\ldots(g^{n_k})^{(t_k)}(\zeta)\equiv p.
 \end{equation}
 Then $g$ is a non vanishing entire function. So using Lemma \ref{lemmaCH} we write  $g(\zeta)=\exp(c\zeta+d),$ where $c(\neq 0), d$ are constants. Then from \eqref{4.2} we get $$(n_1c)^{t_1}\ldots (n_kc)^{t_k}\exp\left(\left(n+\sum_{i=1}^{k}n_i\right)c\zeta+\left(n+\sum_{i=1}^{k}n_i\right)d\right)\equiv p,$$
which is not possible. Hence $g^n(\zeta)(g^{n_1})^{(t_1)}(\zeta)\ldots(g^{n_k})^{(t_k)}(\zeta)\not\equiv p.$\\

Therefore by Lemma \ref{lem3} and Lemma \ref{lem4}, $g^n(\zeta)(g^{n_1})^{(t_1)}(\zeta)\ldots(g^{n_k})^{(t_k)}(\zeta)- p$ has at least two distinct zeros say $\zeta_0  \ \text{and} \ \zeta_0^* $. Now, we choose $\delta>0$ small enough so that \\ $\Delta(\zeta_0,\delta)\cap \Delta(\zeta_{0}^*,\delta)=\emptyset$  and $g^n(\zeta)(g^{n_1})^{(t_1)}(\zeta)\ldots(g^{n_k})^{(t_k)}(\zeta)- p$ has no other zeros in $\Delta(\zeta_0,\delta)\cup \Delta(\zeta_{0}^*,\delta)$. By Hurwitz's theorem, there exist two sequences $\{\zeta_{j}\}\subset \Delta(\zeta_0,\delta),$ \quad $ \{\zeta_{j}^*\}\subset \Delta(\zeta_{0}^*,\delta)$ converging to $\zeta_{0},\ \text{and}\ \zeta_{0}^*$ respectively and from \eqref{4.1}, for sufficiently large $j,$ we have
 \begin{equation}\label{4.3}
g_j^n(\zeta_j)(g_j^{n_1})^{(t_1)}(\zeta_j)\ldots(g_j^{n_k})^{(t_k)}(\zeta_j)=p \ \text{and}\ g_j^n(\zeta_j^*)(g_j^{n_1})^{(t_1)}(\zeta_j^*)\ldots(g_j^{n_k})^{(t_k)}(\zeta_j^*)=p.
\end{equation}

 Since, by assumption that $f_j^n(f_j^{n_1})^{(t_1)}\ldots(f_j^{n_k})^{(t_k)}$ and $f_m^n(f_m^{n_1})^{(t_1)}\ldots(f_m^{n_k})^{(t_k)}$ share $p$  in ${D}=\Delta$, for each pair  $f_j$  and $f_m$ in $\fr$, then  by \eqref{4.3}, for any $m$ and for all $j$ we get \begin{center}$g_m^n(\zeta_j)(g_m^{n_1})^{(t_1)}(\zeta_j)\ldots(g_m^{n_k})^{(t_k)}(\zeta_j)=p \ \text{and}\ g_m^n(\zeta_j^*)(g_m^{n_1})^{(t_1)}(\zeta_j^*)\ldots(g_m^{n_k})^{(t_k)}(\zeta_j^*)=p.$ \end{center}

We fix $m$  and letting $j\rightarrow \ity,$ and noting $z_j+\rho_j\zeta_j\rightarrow 0$, $z_j+\rho_j\zeta_{j}^*\rightarrow 0$, we obtain
\begin{equation}\notag
f_m^n(0)(f_m^{n_1})^{(t_1)}(0)\ldots(f_m^{n_k})^{(t_k)}(0)-p=0.
\end{equation}
Since the zeros  are isolated, for sufficiently large $j$  we have
$z_j+\rho_j\zeta_j=0,$   $z_j+\rho_j\zeta_{j}^*=0 $. Hence  $\zeta_j=-{z_j}/{\rho_j}$ and $\zeta_{j}^*=-{z_j}/{\rho_j},$
which is not possible as $\Delta(\zeta_0,\delta)\cap \Delta(\zeta_{0}^*,\delta)=\emptyset$. This completes the proof. \hfill $\Box$

The proof of Theorem \ref{thm4} is similar to the proof of Theorem \ref{thm1}.\\

{\bf{Proof of Theorem \ref{thm2}}}:
    We may again assume that
${D}=\Delta$. Suppose that $\fr$ is not normal
in $\Delta$. Then there exists at least one point $z_0$ such that
$\fr$ is not normal  at the point $z_0$ in $\Delta$. Without loss of
generality we assume that $z_0=0$. Then by Lemma \ref{lem1}, for \[\al=-\frac{\sum_{i=1}^{k}t_i}{n+\sum_{i=1}^{k}n_i}\]  there
exist
\begin{enumerate}
\item{ a sequence of complex numbers $z_j \rightarrow 0$, $|z_j|<r<1$},
\item{ a sequence of functions $f_j\in \mathcal F$ },
\item{ a sequence of positive numbers $\rho_j \rightarrow 0$},
\end{enumerate}
such that $g_j(\zeta)=\rho_j^{\alpha}f_j(z_j+\rho_j\zeta) $ converges to a non-constant meromorphic function $g(\zeta)$ on $\C$ with $g^{\#}(\zeta)\leq g^{\#}(0)=1$. Moreover $g$ is of order at most two . \\
We see that
\begin{align}
 f_j^n(z_j+\rho_j\zeta)&(f_j^{n_1})^{(t_1)}(z_j+\rho_j\zeta)\ldots(f_j^{n_k})^{(t_k)}(z_j+\rho_j\zeta)\notag\\
&=g_j^n(\zeta)(g_j^{n_1})^{(t_1)}(\zeta)\ldots(g_j^{n_k})^{(t_k)}(\zeta)\rightarrow g^n(\zeta)(g^{n_1})^{(t_1)}(\zeta)\ldots(g^{n_k})^{(t_k)}(\zeta),\label{5.1}
\end{align}
$\text{as} \ j\rightarrow \ity,$ locally spherically uniformly.\\

From the proof of above result we see that $g^n(\zeta)(g^{n_1})^{(t_1)}(\zeta)\ldots(g^{n_k})^{(t_k)}(\zeta)\not\equiv p.$ Now we claim that $g^n(\zeta)(g^{n_1})^{(t_1)}(\zeta)\ldots(g^{n_k})^{(t_k)}(\zeta)-p$ has at most one zero IM. Suppose that $g^n(\zeta)(g^{n_1})^{(t_1)}(\zeta)\ldots(g^{n_k})^{(t_k)}(\zeta)-p$ has two distinct zeros say $\zeta_0$ and $\zeta_0^*$ and choose $\delta>0$ small enough so that $\Delta(\zeta_0,\delta)\cap \Delta(\zeta_{0}^*,\delta)=\emptyset$  and $g^n(\zeta)(g^{n_1})^{(t_1)}(\zeta)\ldots(g^{n_k})^{(t_k)}(\zeta)- p$ has no other zeros in $\Delta(\zeta_0,\delta)\cup \Delta(\zeta_{0}^*,\delta)$. By Hurwitz's theorem, there exist two sequences $\{\zeta_{j}\}\subset \Delta(\zeta_0,\delta), \{\zeta_{j}^*\}\subset \Delta(\zeta_{0}^*,\delta)$ converging to $\zeta_{0}\ \text{and}\ \zeta_{0}^*$ respectively and from \eqref{5.1}, for sufficiently large $j,$ we have
 \begin{equation}\label{5.3}
g_j^n(\zeta_j)(g_j^{n_1})^{(t_1)}(\zeta_j)\ldots(g_j^{n_k})^{(t_k)}(\zeta_j)=p \ \text{and}\ g_j^n(\zeta_j^*)(g_j^{n_1})^{(t_1)}(\zeta_j^*)\ldots(g_j^{n_k})^{(t_k)}(\zeta_j^*)=p.
\end{equation}

Since $z_j\rightarrow 0$ and $\rho_j\rightarrow 0$, we get for sufficiently large $j$, $z_j+\rho_j\zeta_j\in \Delta(\zeta_0,\delta)$ and $z_j+\rho_j\zeta_{j}^*\in \Delta(\zeta_{0}^*,\delta)$. Therefore $f_j^n(f_j^{n_1})^{(t_1)}\ldots(f_j^{n_k})^{(t_k)}-p $ has two distinct zeros, which contradicts the fact that $f_j^n(f_j^{n_1})^{(t_1)}\ldots(f_j^{n_k})^{(t_k)}-p$ has at most one zero. But  Lemma \ref{lem3} and Lemma \ref{lem4} confirm that there does not exist  such a non-constant  meromorphic function. This contradiction shows that $\fr$ is normal in $\Delta$ and this proves the theorem. \hfill $\Box$

\section{Extensions of Theorem \ref{thm1} and Theorem \ref{thm4}}

It is natural to ask whether one can replace the value $p$, in Theorem \ref{thm1}  by a holomorphic function $\alpha(z)$. In this direction  we extend Theorem \ref{thm1} in the following manner:

\begin{theorem}\label{thm3}
Let $\al(z)$ be a holomorphic function defined in a domain $D\subset \C$ such that $\al(z)\not=0$.  Let $n$ be a non negative integer and $n_1, n_2,\ldots, n_k, t_1, t_2,\ldots, t_k $ be  positive integers such that
  \begin{enumerate}
 \item[$(a)$]{ $n_j \geq t_j$ for all $1\leq j \leq k,$}
 \item [$(b)$]{$n+\sum_{j=1}^{k}n_j\geq 3+ \sum_{j=1}^{k}t_j.$}
 \end{enumerate}
Let $\fr$ be a family of meromorphic functions in a domain $D$ such that for every pair of functions $f, g \in \fr,  f^n(z)(f^{n_1})^{(t_1)}(z)\ldots(f^{n_k})^{(t_k)}(z)$ and $g^n(z)(g^{n_1})^{(t_1)}(z)\ldots(g^{n_k})^{(t_k)}(z)$ share $\al(z)$ IM on $D$. Then $\fr$ is normal in $D$.\\
 \end{theorem}
\begin{proof}

Once again we assume that
${D}=\Delta$. Suppose that $\fr$ is not normal
in $\Delta$. Then there exists at least one point $z_0$ such that
$\fr$ is not normal  at the point $z_0$ in $\Delta$. Without loss of
generality we assume that $z_0=0$. Then by Lemma \ref{lem1}, for \[\al=-\frac{\sum_{i=1}^{k}t_i}{n+\sum_{i=1}^{k}n_i}\]  there
exist
\begin{enumerate}
\item{ a sequence of complex numbers $z_j \rightarrow 0$, $|z_j|<r<1$},
\item{ a sequence of functions $f_j\in \mathcal F$ },
\item{ a sequence of positive numbers $\rho_j \rightarrow 0$},
\end{enumerate}
such that $g_j(\zeta)=\rho_j^{\alpha}f_j(z_j+\rho_j\zeta) $ converges to a non-constant meromorphic function $g(\zeta)$ on $\C$ with $g^{\#}(\zeta)\leq g^{\#}(0)=1$. Moreover $g$ is of order at most two . \\
We see that
\begin{align}
f_j^n(z_j+\rho_j\zeta)&(f_j^{n_1})^{(t_1)}(z_j+\rho_j\zeta)\ldots(f_j^{n_k})^{(t_k)}(z_j+\rho_j\zeta)-\al(z_j+\rho_j\zeta)\notag\\
&=g_j^n(\zeta)(g_j^{n_1})^{(t_1)}(\zeta)\ldots(g_j^{n_k})^{(t_k)}(\zeta) - \al(z_j+\rho_j\zeta)\notag\\
&\rightarrow g^n(\zeta)(g^{n_1})^{(t_1)}(\zeta)\ldots(g^{n_k})^{(t_k)}(\zeta)- \al(0),\label{6.1}
\end{align}
$\text{as} \ j\rightarrow \ity ,$ locally spherically uniformly.\\

Let
\begin{equation}\label{6.2}
g^n(\zeta)(g^{n_1})^{(t_1)}(\zeta)\ldots(g^{n_k})^{(t_k)}(\zeta)\equiv \al(0).
 \end{equation}
 Then $g$ is an entire function having no zero. So by Lemma \ref{lemmaCH} we write $g(\zeta)=\exp(c\zeta+d),$ where $c(\neq 0), d$ are constants. Then from \eqref{6.2} we get $$(n_1c)^{t_1}\ldots (n_kc)^{t_k}\exp\left(\left(n+\sum_{i=1}^{k}n_i\right)c\zeta+\left(n+\sum_{i=1}^{k}n_i\right)d\right)\equiv \al(0),$$
which is not possible. Hence $g^n(\zeta)(g^{n_1})^{(t_1)}(\zeta)\ldots(g^{n_k})^{(t_k)}(\zeta)\not\equiv \al(0).$\\

Therefore by Lemma \ref{lem3} and Lemma \ref{lem4} $g^n(\zeta)(g^{n_1})^{(t_1)}(\zeta)\ldots(g^{n_k})^{(t_k)}(\zeta)- \al(0)$ has at least two distinct zeros say $\zeta_0  \ \text{and} \ \zeta_0^* $. Now, we choose $\delta>0$ small enough so that $\Delta(\zeta_0,\delta)\cap \Delta(\zeta_{0}^*,\delta)=\emptyset$  and $g^n(\zeta)(g^{n_1})^{(t_1)}(\zeta)\ldots(g^{n_k})^{(t_k)}(\zeta)- \al(0)$ has no other zeros in $\Delta(\zeta_0,\delta)\cup \Delta(\zeta_{0}^*,\delta)$.\\

 By Hurwitz's theorem, there exist two sequences $\{\zeta_{j}\}\subset \Delta(\zeta_0,\delta), \{\zeta_{j}^*\}\subset \Delta(\zeta_{0}^*,\delta)$ converging to $\zeta_{0},\ \text{and}\ \zeta_{0}^*$ respectively and from \eqref{6.1}, for sufficiently large $j$, we have
 \begin{align}
&g_j^n(\zeta_j)(g_j^{n_1})^{(t_1)}(\zeta_j)\ldots(g_j^{n_k})^{(t_k)}(\zeta_j)=\al(z_j+\rho_j\zeta_j)\notag  \\
&g_j^n(\zeta_j^*)(g_j^{n_1})^{(t_1)}(\zeta_j^*)\ldots(g_j^{n_k})^{(t_k)}(\zeta_j^*)=\al(z_j+\rho_j\zeta_j).\label{6.3}
\end{align}

 Since, by assumption that $f_j^n(f_j^{n_1})^{(t_1)}\ldots(f_j^{n_k})^{(t_k)}$ and $f_m^n(f_m^{n_1})^{(t_1)}\ldots(f_m^{n_k})^{(t_k)}$ share $\al(z)$ IM  in ${D}=\Delta$, for each pair  $f_j$  and $f_m$ in $\fr$, then  by \eqref{6.3}, for any $m$ and for all $j$ we get $$g_m^n(\zeta_j)(g_m^{n_1})^{(t_1)}(\zeta_j)\ldots(g_m^{n_k})^{(t_k)}(\zeta_j)=\al(z_j+\rho_j\zeta_j)$$ and $$g_m^n(\zeta_j^*)(g_m^{n_1})^{(t_1)}(\zeta_j^*)\ldots(g_m^{n_k})^{(t_k)}(\zeta_j^*)=\al(z_j+\rho_j\zeta_j).$$\\

We fix $m$  and letting $j\rightarrow \ity,$ and noting $z_j+\rho_j\zeta_j\rightarrow 0$, $z_j+\rho_j\zeta_{j}^*\rightarrow 0$, we obtain
\begin{equation}\notag
f_m^n(0)(f_m^{n_1})^{(t_1)}(0)\ldots(f_m^{n_k})^{(t_k)}(0)-\al(0)=0.
\end{equation}
Since the zeros are isolated, for sufficiently large $j$  we have
$z_j+\rho_j\zeta_j=0,$   $z_j+\rho_j\zeta_{j}^*=0 $. Hence $\zeta_j=-{z_j}/{\rho_j},$ and $ \zeta_{j}^*=-{z_j}/{\rho_j}$,
which is not possible as $\Delta(\zeta_0,\delta)\cap \Delta(\zeta_{0}^*,\delta)=\emptyset$. This completes the proof.
\end{proof}

For  families of holomorphic functions  we have the following result:
\begin{theorem}\label{thm5}
Let $\al(z)$ be a holomorphic function defined in a domain $D\subset \C$ such that $\al(z)\not=0$.  Let $n$ be a non negative integer and $n_1, n_2,\ldots, n_k, t_1, t_2,\ldots, t_k $ be  positive integers such that
  \begin{enumerate}
 \item[$(a)$]{ $n_j \geq t_j$ for all $1\leq j \leq k,$}
 \item [$(b)$]{$n+\sum_{j=1}^{k}n_j\geq 2+ \sum_{j=1}^{k}t_j.$}
 \end{enumerate}
Let $\fr$ be a family of holomorphic functions in a domain $D$ such that for every pair of functions $f, g \in \fr,  f^n(z)(f^{n_1})^{(t_1)}(z)\ldots(f^{n_k})^{(t_k)}(z)$ and $g^n(z)(g^{n_1})^{(t_1)}(z)\ldots(g^{n_k})^{(t_k)}(z)$ share $\al(z)$ IM on $D$. Then $\fr$ is normal in $D$.
 \end{theorem}
The proof is similar to  the proof of Theorem \ref{thm3}\\

\textbf{Acknowledgements} \quad  We wish to thank Professor Kaushal  Verma, IISc Bangalore for reading the preliminary version of the manuscript. His suggestions at the numerous steps in the paper has improved considerably the mathematical contents of the paper.


\begin{thebibliography}{00}

\bibitem{Ahl}  Ahlfors LV. {Complex Analysis,} Third edition, {McGraw-Hill,} 1979.

\bibitem{CH}  Clunie J,  Hayman WK. {The spherical derivative of integral and meromorphic function,} Comment Math Helv 1966; 40:  117--148.
\bibitem{GS} Datt G,  Kumar S. {Normality and Sharing functions,} (accepted for publication) {Ind J Pur Appl Math}, arXiv: 1308.5401.
\bibitem{dethloff}  Dethloff G, Tan TV, Thin NV. {Normal criteria for  families of meromorphic functions,} J Math Anal Appl 2014;  411:  675--683.
\bibitem{HG}  Han M,  Gu Y. {The normal family of meromophic functions,}{ Acta Math Sci} 2008; {28 B} No. 4:   759--762.
\bibitem{Hay}  Hayman WK. {Meromorphic Functions,} {Clarendon Press, Oxford}, 1964.
\bibitem{hin}  Hinchliffe JD. {On a result of Chuang related to Hayman's alternative,} Comput Methods Funct Theory 2002; 2: 293--297.
\bibitem{Li}  Li Y,  Gu Y. {On normal families of meromorphic functions}, J Math Anal Appl 2009;  {354}:  421--425.
\bibitem{Schiff}  Schiff JL. { Normal Families,}  Berlin: Springer-Verlag,  1993.
\bibitem{Sch1} Schwick W. {Normality criteria for families of meromophic functions,} J Anal Math 1989; 52: 241--289.
\bibitem{Sch} Schwick W. {Sharing values and normality,} {Arch der Math} 1992; {59}: 50--54.
\bibitem{YMW}  Wang YM.  {On normality of meromorphic functions with multiple zeros and sharing values}, Turk J Math 2012; 36:  263--271.
\bibitem{ccyang}  Yang CC,  Yi HX.  {Uniqueness theory of meromorphic functions,} {Science Press/ Kluwer Academic Publishers,} 2003.

\bibitem{Yang} Yang L.  {Value Distribution Theory,}  Berlin: {Springer-Verlag,} 1993.
\bibitem{Zalc 1} Zalcman L. {A heuristic principle in complex function theory,} The Amer Math Monthly 1975; {82}: 813--817.
\bibitem{Zalc} Zalcman L. {Normal families: new perspectives,} Bull Amer Math Soc 1998; {35}, no. 3:  215--230.
\bibitem{lahiri} Zeng S,  Lahiri I. {A normality criterion for meromorphic functions,} Kodai Math J 2012; {35}: 105--114.
\bibitem{lahiri 1} Zeng S,  Lahiri I. {A normality criterion for meromorphic functions having multiple zeros,} Ann Polon Math 2014; 110.3:  283--294.
\end{thebibliography}
\end{document}